\documentclass[11pt,reqno]{amsart}
\usepackage{graphicx}
\pdfoutput=1
\usepackage{amsfonts,amsmath,amssymb}
\usepackage{hyperref}
\usepackage{paralist}
\usepackage[linesnumbered,ruled,vlined,norelsize]{algorithm2e}

\usepackage{thmtools}
\declaretheoremstyle[bodyfont=\normalfont]{noncursive}
\declaretheorem{theorem}
\declaretheorem[numberwithin=section]{lemma}

\declaretheorem[numberlike=lemma]{proposition}

\declaretheorem[style=noncursive,numberlike=lemma]{definition}
\declaretheorem[style=noncursive,numberlike=lemma]{example}
\declaretheorem[style=noncursive,numberlike=lemma]{remark}

\newcommand{\im}{\ensuremath{\mbox{\rm Im}\,}}

\newcommand{\CC}[1]{\mathbb{C}^{#1}}

\newcommand{\RR}[1]{\mathbb{R}^{#1}}

\newcommand{\ph}{\varphi}

\numberwithin{equation}{section}

\newcommand{\tc}[1]{T^{\mathbb C}_{#1}}
\def\Label#1{\label{#1}}

\def\1#1{\overline{#1}}
\def\2#1{\widetilde{#1}}
\def\3#1{\widehat{#1}}
\def\4#1{\mathbb{#1}}
\def\5#1{\frak{#1}}
\def\6#1{{\mathcal{#1}}}

\def\C{{\4C}}

\title[Real-analytic coordinates]{Real-analytic coordinates for smooth strictly pseudoconvex CR-structures}

\author {I. Kossovskiy}
\address{Department of Mathematics, Masaryk University in Brno//
Faculty of Mathematics, University of Vienna}
\email{kossovskiyi@math.muni.cz, ilya.kossovskiy@univie.ac.at}
\author {D. Zaitsev}
\address{School of Mathematics, Trinity College, Dublin}
\email{zaitsev@maths.tcd.ie}

\begin{document}

\maketitle

\date{\today}

\begin{abstract}
For a smooth strictly pseudoconvex hypersurface in a complex manifold, 
we give a necessary and sufficient condition for being CR-diffeomorphic
to a real-analytic CR manifold. Our condition amounts to a holomorphic extension property for the canonically associated function expressing
$2$-jets of the formal Segre varieties in terms of their $1$-jets.
We also express this condition in equivalent terms for a Fefferman type determinant \cite{feffer}. 
\end{abstract}

\tableofcontents

\section{Introduction}
In this paper, we address the following problem:

\smallskip

\noindent{\bf Problem 1.}
Let $M$ be a smooth real hypersurface in a complex manifold $X$. 
Find necessary and sufficient conditions on $M$
to be CR-diffeomorphic to a real-analytic CR manifold.

\smallskip

If $M$ is CR-diffeomorphic to a real-analytic CR manifold,
we shall call it 
{\em analytically regularizable}.
Problem 1 is of interest because of
the case of real-analytic CR manifolds
being much better studied
with more results and tools available
such as complexification and Segre varieties
(see e.g.\  \cite{ber}).
On the other hand, 
the problem
seems to be widely open
even for strictly pseudoconvex  hypersurfaces,
where it is non-trivial,
i.e.\ there exist smooth non-analytic hypersurfaces 
that are analytically regularizable and there exist those that are not,
see the end of Section 2 for respective examples.
(The latter phenomenon is in contrast with the case of 
hypersurfaces of mixed Levi form signature,
where any CR-diffeomorphism to a real-analytic hypersurface
extends holomorphically to both sides,
hence any analytically regularizable hypersurface must be
already real-analytic in its ambient complex manifold $X$.)

The goal of this paper is 
to provide 
a nontrivial {\em necessary and sufficient} condition (Condition E below)
giving a solution to Problem 1
in the case
when $M$ is {\em strictly pseudoconvex}.
Our condition is formulated
 in terms of holomorphic extension of certain 
functions invariantly associated to $M$ (these functions can be viewed as Fefferman type determinants \cite{feffer}, as shown in Section 2).
Thus our results provide a {\em two-way bridge} for Problem 1
with the much better studied 
questions of holomorphic extension of smooth functions from real submanifolds in complex manifolds.


We proceed by giving a high level non-technical
formulation of our result,
while more precise details can be found 
in Section 2. 
Given a smooth real hypersurface $M$
in a complex manifold $X$ of dimension $n+1$,
consider smooth local defining equations
of the kind $\rho(z, \bar z)=0$ for $M$ with $d\rho\ne 0$.
Then $\rho$ can be formally complexified at each point,
i.e.\ there exist formal power series $\rho(z,\bar w)$
giving for $w=z$ the Taylor series of $\rho$.
This allows to invariantly define formal
Segre varieties $$Q_p=\{z: \rho(z, \bar p)=0\}$$ for $p\in M$,
as well as their $k$-jets $j^k_pQ_p$ for every $k\ge1$.
In particular (see \cite{websterrefl}), 
if $M$ is strictly pseudoconvex,
then $j^kQ_p$, $p\in M$,
defines a canonical embedding of $M$
into the space $J^{k,n}(X)$ of all $k$-jets of complex-analytic
hypersurfaces for every $k\ge 1$,
given by $$p\mapsto j^kQ_p,$$ and the image of $M$ appears to be totally real.

Furthermore, for any $k,l\ge 1$,
we obtain canonical smooth maps
$s^{k,l}$
between respective images of those embeddings of $M$,
sending  $j^kQ_p$ to  $j^lQ_p$ for every $p$.
To formulate our main result,
denote by $M_J\subset J^{1,n}$
the image of the embedding $p\mapsto j^1Q_p$ of $M$ as above,
and consider the map 
$$s^{1,2}\colon M_J\to J^{2,n}$$
sending $j^1_pQ_p$ to  $j^2_pQ_p$
for every $p\in M$.
Let $\3M\subset J^{1,n}$ be
the (smooth) real  hypersurface
consisting of all $1$-jets with base points in $M$.

\begin{theorem}\Label{main0} 
Let $M$ be a smooth strictly pseudoconvex hypersurface
in a complex manifold $X$.
Then $M$ is analytically regularizable
(i.e.\ CR-diffeomorphic to a real-analytic  CR manifold)
if and only if the map $s^{1,2}$
admits a local holomorphic extension
to the pseudoconvex side of 
$\3M$ in $J^{1,n}$ valued in $J^{2,n}$
that is smooth up to $\3M$.
\end{theorem}

Note that since $M_J$ is generic in $J^{1,n}$,
if such holomorphic extension of $s^{1,2}$ exists,
it is necessarily unique.
Also, the CR-diffeomorphism to a real-analytic CR manifold
is unique if exists,
up to a real-analytic CR-diffeomorphism,
 as follows directly from the reflection principle for real-analytic strictly pseudoconvex hypersurfaces.

We conclude by mentioning that, 
while analyticity problems have been 
studied for other geometric structures, e.g. Riemannian structures
(see e.g.\ \cite{kazdan}),
no similar nontrivial necessary and sufficient conditions
seem to be known.


\medskip

\begin{center} 
\bf Acknowledments
\end{center}

\medskip
We would like to thank Laszlo Lempert for his helpful remarks 
following a talk given by the first author on the subject.

The first author was supported by the GACR (Czech Grant Agency) and the FWF (Austrian Science Fund) during the preparation of this paper.

\section{Condition E for smoothly embedded real hypersurfaces}

 We now describe in details the holomorphic extension condition
 in  \autoref{main0}
 that we shall call
  Condition E.
 We shall give below both an invariant and a coordinate-based formulations of it. 
 For the basic concepts in CR-geometry (such as Segre varieties and formal submanifolds) we refer to \cite{ber}, and for {\em jet bundles} and related concepts to \cite{cs}. 

Let 
$$
	\pi \colon J^{1,n}
	\to \CC{n+1}
$$ 
be the bundle
of $1$-jets of complex hypersurfaces of $\CC{n+1}$, 
which is a projective holomorphic   bundle over $\CC{n+1}$ with the fiber dimension $n$,
and
$M\subset\CC{n+1},\,n\geq 1$, 
be a smooth strictly pseudoconvex real hypersurface. 
Then the complex tangent bundle $\tc{}M$ induces
the natural embedding 
$$
	\varphi\colon M\to J^{1,n}, \quad x \mapsto \bigl(x,[\tc{x}M]\bigr).
$$ 
The image 
$$
	\ph(M)=:M_J \subset J^{1,n}
$$ 
is a  smooth $(2n+1)$-dimensional real submanifold  in the $(2n+1)$-dimensional complex manifold $J^{1,n}$. Webster in \cite{websterrefl} observed that 
$M_J\subset J^{1,n}$ is {\em totally real} whenever $M$ is Levi-nondegenerate.
Next, 
associated with $M$ is the  smooth pseudoconvex real hypersurface
$$
	\widehat M:=\pi^{-1}(M) \subset J^{1,n}.
$$
The manifold $M_J$ is a smooth submanifold in $\widehat M$.  
Note that  $\widehat M$ itself is locally CR-equivalent 
to $M\times \CC{n}$ (and thus is {\em holomorphically degenerate}, see \cite{ber}).  
In what follows we denote by $U^+$ the pseudoconvex side of $M$ and by $$\widehat U^+:=\pi^{-1}(U^+)$$ that of $\widehat M$ ($\widehat U^+$ is locally biholomorphic to $U^+\times\CC{}$ at a point in $\widehat M$).

We next fix a point $p\in M$.
Since $M$ is smooth, we may consider at each point $q\in M$ 
near $p$, its formal complexfication at $q$ as a formal complex hypersurface in $\CC{n+1}\times\overline{\CC{n+1}}$ obtained by complexifying the formal Taylor series of its defining function at $q$. In this way, the formal Segre variety $S_q$ of $M$ at $q$ is well defined. Then the  $2$-jets 
\begin{equation}\Label{2jets}
j^2_qS_q,\,\,q\in M
\end{equation} of such formal Segre varieties induce a smooth embedding  
of $M$ (and hence $M_J\subset J^{1,n}$) into the bundle  
$$J^{2,n}=J^{2,n}(\CC{n+1})$$ 
of $2$-jets of complex hypersurfaces in $\CC{n+1}$. The space $J^{2,n}$  is canonically a fiber bundle  
$$
\pi^2_1 \colon J^{2,n}\to J^{1,n}.
$$
The above  $2$-jet embedding defines a canonical section of $\pi^2_1$, 
$$s\colon M_J\to J^{2,n}.$$ 

Now our analyticity condition for a  smooth strictly pseudoconvex submanifold looks as follows.

\begin{definition}\Label{coorfree} We say that $M$ {\em satisfies Condition E at $p$}, if  for some choice of a neighborhood $U$ of $p$, the section $s$ extends as a smooth section of $\pi^2_1$ over the pseudoconvex side $\widehat U^+ \cup \widehat M$, which is  furthermore holomorphic in $\widehat U^+$.
\end{definition}


We next give an (equivalent to the above) coodinate formulation of Condition E. If $M\subset\CC{n+1}$ is a smooth hypersurface with the defining equation 
\begin{equation}\Label{rhoeq}
\rho(Z,\bar Z)=0,\quad Z=(z,w)=(z_1,...,z_n,w)\in\CC{n+1},
\end{equation}
$p\in M$ the distinguished point and $\rho_w(p,\bar p)\neq 0$, then its formal Segre variety at a point $q=(\tilde q,q_{n+1})\in M$ nearby $p$ is a graph of a function $w(z)$ (considered as a formal power series in $(z-\tilde q)$). Then the $2$-jets \eqref{2jets} amount  to either the scalar function 
$\Phi$ defined pointwise as $w''(z)$ for $z=\tilde q$  (case $n=1$),  or to the symmetric matrix function $\Phi=(\Phi_{ij}),\,\,i,j=1,...,n$, defined pointwise as the collection of  $w_{z_iz_j}$ for $z=\tilde q$ (case $n>1$). It is possible to verify that, in turn, for $n=1$ we have
\begin{equation}\Label{Phi1}
	\Phi=\frac{1}{(\rho_w)^3}
	\begin{vmatrix} \rho & \rho_z & \rho_w\\ \rho_z & \rho_{zz} 
	& \rho_{zw} \\ \rho_w & \rho_{zw} & \rho_{ww} \end{vmatrix},
\end{equation}
and for $n>1$ we have 
\begin{equation}\Label{Phi}
	\Phi_{ij}=\frac{1}{(\rho_w)^3}
\begin{vmatrix} \rho & \rho_{z_j} & \rho_w\\ \rho_{z_i} & \rho_{z_iz_j} & \rho_{z_iw} \\ \rho_w & \rho_{z_jw} & \rho_{ww}, \end{vmatrix} \quad i,j=1,...,n.
\end{equation}
 (To obtain \eqref{Phi1},\eqref{Phi}, one has to differentiate twice the identity \eqref{rhoeq} assuming $w$ to be a function of $z$). Both the scalar function \eqref{Phi1} and the matrix valued function \eqref{Phi} can be considered as either smooth functions on the strictly pseudoconvex hypersurface $M$ or as that on the totally real manifold $M_J$ introduced above. 
 
 We shall remark that the invariant determinants in \eqref{Phi1},\eqref{Phi} were used also by Ebenfelt and the second author \cite{ez}
 as well as by Ebenfelt, Duong and the second author
 \cite{edz} for characterizing the Cartan tensor of a Levi-degenerate hypersurface in terms of its defining function. They can be seen as certain generalizations of the determinants used by Fefferman in \cite{feffer} for studying asymptotics of the Bergman metric in a smoothly bounded strictly pseudoconvex domain at a boundary point.
 
In terms of the $\Phi$-function, Condition E reads as follows.
\begin{definition}\Label{coor} We say that $M$ {\em satisfies Condition E at $p$}, if  for some choice of a neighborhood $U$ of $p$, the function $\Phi$ defined on $M_J$ by either \eqref{Phi1} or \eqref{Phi} extends  to the pseudoconvex side $\widehat U^+ \cup \widehat M$ holomorphically and smoothly up to the boundary.
\end{definition}
It is obvious that \autoref{coor} is equivalent to \autoref{coorfree}. 

We give now the more precise local version
of our main result, from which the global result in \autoref{main0}
follows directly in view of the uniqueness
of the extension and the real-analytic CR-structure.

\begin{theorem}\Label{main} 
A  smooth strictly pseudoconvex  real hypersurface $M\subset\CC{n+1},\,n\geq 1$, is locally analytically regularizable
(i.e.\ CR-equivalent near a point $p\in M$ to a real-analytic  hypersurface $\2M\subset\CC{n+1}$) if and only it satisfies  Condition E at $p$.
\end{theorem}

\begin{remark}\Label{Ck}
In fact, the above Condition E for  smooth hypersurfaces can be naturally extended to  hypersurfaces of merely class $C^2$. Following then the details of the proof of \autoref{main} below, it can be seen that, {\em for a $C^k$ smooth strictly pseudoconvex hypersurface $M\subset\CC{n}$ with $k\geq 2$,  Condition E is necessary and sufficient for the existence of a local $C^{k-1}$ smooth CR-diffeomorhism of $M$ onto a real-analytic hypersurface $\tilde M\subset\CC{n+1}$.}
\end{remark}

We finish this section by providing two simple examples of 
smooth non-analytic strictly pseudoconvex hypersurfaces in $\CC{2}$, 
one of which admits and the other does not a CR-diffeomorphism onto a real-analytic hypersurface,
i.e.\ one is analytically regularizable while the other isn't.

\begin{example}
Let $f(z,w)$ be a holomorphic function in the unit ball $\mathbb B^2\subset\CC{2}$ which is smooth up to $\partial\mathbb B^2=S^3$ but does not extend holomorphically across $S^3$, e.g. one can take 
a branch of $e^{(w-1)^{-1/3}}$.
Then, for $\epsilon$ sufficiently small, the map
(small perturbation of identity)
$$
	(z,w)\mapsto (z + \epsilon f(z,w),w)
$$ 
defines a CR-diffeomorphism from $S^3$ onto a smooth
but not analytic strictly pseudoconvex hypersurface $M\subset\CC{2}$
that is obviously analytically regularizable.
\end{example}

\begin{example}
The real hypersurface $M\subset\CC{2}$ 
(a flat perturbation at the origin of the standard hyperquadric) 
given by 
$$
	\im w=|z^2|+e^{-1/|z|^2}
$$
is smooth and strictly pseudoconvex near the origin, 
but is not CR-diffeomorphic to a real-analytic hypersurface 
$\tilde M\subset \C^2$, not even locally at $0$.
Indeed, since $M$ is {\em formally spherical} at $0$
(i.e.\ spherical up to infinite order),
 the existence of a CR-diffeomorphism $H$ onto $\tilde M$ with, say, $H(0)=0$ would mean that $\tilde M$ is also formally
 and hence biholomorphically spherical at $0$.
(The formal expansion of $H$ at $0$ yields a formal transformation of the hyperquadric $\im w=|z|^2$ onto $\tilde M$
which is locally biholomorphic in view of the Chern-Moser theory \cite{chern}.)
Hence $M$ must itself be spherical
in a neighborhood of $0$. 
On the other hand, 
one can see by computing the Chern-Moser's curvature 
(e.g. following \cite{lobumbilic} or using 
the determinant expression \cite{ez}), that $M$ is not spherical in any neighborhood of $0$. 
Thus $M$ is not analytically regularizable
in any neighborhood of $0$.
\end{example}

  \section{Associated differential equations and the necessity of Condition E} 
  
In this section, we show that the necessity of Condition E follows rather easily from the construction of {\em holomorphic differential equations associated with a real-analytic hypersurface}. On the other hand, the sufficiency of Condition E is already quite nontrivial. It is addressed in Section 4.
  
 \subsection{The method of associated differential equations} It was observed by Cartan \cite{cartan} and Segre \cite{segre} (see also Webster \cite{webster}) that the geometry of a real hypersurface in $\CC{2}$ parallels that of a second order ODE
 \begin{equation}\Label{wzz}
 w\rq{}\rq{}=\Phi(z,w,w\rq{}).
 \end{equation}
More generally,  the geometry of a real hypersurface in $\CC{n+1},\,n\geq 1$, parallels that of a complete second order system of PDEs
\begin{equation}\Label{wzkzl}
w_{z_kz_l}=\Phi_{kl}(z_1,...,z_n,w,w_{z_1},...,w_{z_n}),\quad \Phi_{kl}=\Phi_{lk},\quad k,l=1,...,n.
\end{equation}
  Moreover, {\em in the real-analytic case} this parallel becomes algorithmic by using  the Segre family of a real hypersurface. With any real-analytic Levi-nondegenerate hypersurface $M\subset\CC{n+1},\,n\geq 1$ one can uniquely associate a holomorphic ODE \eqref{wzz} ($n=1$) or a holomorphic PDE system \eqref{wzkzl} ($n\geq 2$). The Segre family of $M$ plays a role of a mediator between the hypersurface and the associated differential equations.  A more recent exposition of this method was given in the work \cite{sukhov1,sukhov2} of Sukhov. For recent work on associated differential equations in the degenerate setting, see e.g. the papers  
  \cite{divergence, nonminimalODE, nonanalytic} of the first author with Lamel and Shafikov.

  The associated differential equation procedure is particularly clear in the case of a Levi-nondegenerate hypersurface in $\CC{2}$. In this case the Segre family is a 2-parameter anti-holomorphic family of  holomorphic curves. It then follows from standard ODE theory that there exists an unique ODE \eqref{wzz}, for which the Segre varieties are precisely the graphs of solutions. This ODE is called \it the associated ODE. \rm

In the general case, both right hand sides in \eqref{wzz},\eqref{wzkzl} appear as functions determining the $2$-jet of a Segre variety as an analytic function of the $1$-jet. More explicitly, we denote the coordinates in $\CC{n+1}$ by 
$$
	(z,w)=(z_1, \ldots, z_n,w).
$$ 
Let then fix $M\subset\CC{n+1}$ to be a smooth real-analytic
hypersurface, passing through the origin, and choose a small neighborhood $U$
 of the origin. In this case
we associate a complete second order system of holomorphic PDEs to $M$,
which is uniquely determined by the condition that the differential equations are satisfied by all the
graphing functions $h(z,\zeta) = w(z)$ of the
Segre family $\{Q_\zeta\}_{\zeta\in U}$ of $M$ in a
neighbourhood of the origin.
To be more explicit we consider the
so-called {\em  complex defining
 equation } (see, e.g., \cite{ber})\,
$w=\rho(z,\bar z,\bar w)$ \, of $M$ near the origin, which one
obtains by substituting $u=\frac{1}{2}(w+\bar
w),\,v=\frac{1}{2i}(w-\bar w)$ into the real defining equation and
applying the holomorphic implicit function theorem.
 The Segre
variety $Q_p$ of a point 
$$x=(a,b)\in U,\,a\in\CC{n},\,b\in\CC{}$$ 
is  now given
as the graph
\begin{equation} \Label{segredf}w (z)=\rho(z,\bar a,\bar b). \end{equation}
Differentiating \eqref{segredf} we obtain
\begin{equation}\Label{segreder} 
	w_{z_j}=\rho_{z_j}(z,\bar a,\bar b),
	\quad
	j=1,\ldots,n. 
\end{equation}
Considering \eqref{segredf} and \eqref{segreder}  as a holomorphic
system of equations with the unknowns $\bar a,\bar b$, 
in view the Levi-nondegeneracy of $M$,
an
application of the implicit function theorem yields holomorphic functions
 $A_1,...,A_n, B$ such that
 \eqref{segredf} and \eqref{segreder} are solved by
$$
	\bar a_j=A_j(z,w,w'),\quad
	\bar b=B(z,w,w'),
$$
where we write
$$
	w' = (w_{z_1},  \ldots, w_{z_n}).
$$
The implicit function theorem applies here because the
Jacobian of the system coincides with the Levi determinant of $M$
for $(z,w)\in M$ (\cite{ber}). Differentiating \eqref{segredf} twice
and substituting the above solution for $\bar a,\bar b$ finally
yields
\begin{equation}\Label{segreder2}
w_{z_kz_l}=\rho_{z_kz_l}(z,A(z,w,w'),
B(z,w,w'))=:\Phi_{kl}(z,w,w'),
\quad
k,l=1, \ldots, n,
\end{equation}
or, more invariantly,
\begin{equation}\Label{segreder2'}
	j^2_{(z,w)} Q_x = \Phi(x, j^1_{(z,w)} Q_x).
\end{equation}
Now \eqref{segreder2} is the desired complete system of holomorphic second order PDEs
denoted by $\mathcal E = \mathcal{E}(M)$.
 \begin{definition}\Label{PDEdef}
 We call $\mathcal E = \mathcal{E}(M)$  \it the system of PDEs 
 associated with $M$. \rm  
  We also regard the collection   $\{\Phi_{ij}\}_{i,j=1}^n$ 
  as {\em the PDE system defining the CR structure} of a Levi-nondegenerate hypersurface $M$.
\end{definition}

\subsection{The necessity of Condition E}
We now explain the necessity of Condition E for the existence of a  smooth CR-diffeomorphism $F$ of $(M,0)$ onto a real-analytic germ $(\tilde M,0)$. Indeed, given such a CR-diffeomorphism $F$ of $(M,0)$ onto $(\tilde M,0)$, we may consider the section $\tilde \Phi$, as in \eqref{segreder2'}, associated with $(\tilde M,0)$. Clearly, Condition E is satisfied by $(\tilde M,0)$ since $\tilde \Phi$, considered as a function on the $1$-jet bundle $J^{1,n}$, already gives the holomorphic extension required in Condition E. Further, we note that the CR-diffeomorphism $F$ extends holomorphically to the pseudoconvex side $U^+$. The latter extension lifts naturally to a fiber-preserving map $\widehat F$ of the pseudoconvex neighborhood $\widehat U^+$ of $\widehat M$ into $J^{1,n}$ which is smooth up to $\widehat M_J$ ($\widehat F$ is the {\em $1$-jet prolongation} of the extension of $F$, see e.g. \cite{cs}). Now, since $F$ transforms formal complexifications of $M,\tilde M$ respectively onto each other, we conclude that the $2$-jet prolongation of the extension of $F^{-1}$ transforms $\tilde \Phi$  into the desired holomorphic extension $\Phi$, as required.

\qed

\section{The sufficiency of Condition E}

In this section, we consider a smooth strictly pseudoconvex hypersurface $M\subset\CC{n+1},\,n\geq 1$, defined near the point $0\in M$ and satisfying condition E. We shall prove that $M$ is CR-diffeomorphic (locally near $0$) to a real-analytic hypersurface $\tilde M\subset\CC{n+1}$. 

\subsection{Segre foliation in the space of $1$-jets} 
We start by recalling that the affine subset
$$E\simeq\CC{n+1}\times\CC{n}$$ of the 
bundle $J^{1,n}\to \CC{n+1}$
of $1$-jets of complex hypersurfaces
 is endowed with the canonical
 (up to a scalar function multiple) 
 $1$-form $$\omega_0:=dw-\sum_{1}^n\xi_jdz_j.$$ Here $(z_1,...,z_n,w)=(z,w)$ denote the coordinates in $\CC{n+1}=\CC{n}\times\CC{}$, and $\xi=(\xi_1,...,\xi_n)$ are the respective ``jet''-variables corresponding to the derivatives $w_{z_1},...,w_{z_n}$ respectively. The restriction to $E$ of the canonical projection 
 $\pi:\,J^{1,n}\mapsto \CC{n+1}$ then becomes
\begin{equation}\Label{PI}
\pi:\,\,(z,w,\xi)\mapsto (z,w)
\end{equation} 
and $E$ consists precisely
of the $1$-jets of hypersurfaces
that project submersively onto
 $\CC{n}\times\{0\} \subset \CC{n+1}$. 
 
 The main use of the canonical form $\omega_0$ here is the following. The (complex) tangent bundle $TS$ of a complex hypersurface $S\subset\CC{n+1}$ given as a graph of a function $w=w(z)$ allows to naturally lift $S$ to a complex $n$-dimensional submanifold of $E$. Then an $n$-dimensional submanifold $\tilde S\subset E$ of the kind $$w=w(z),\,\xi=\xi(z)$$ is a lifting in the above sense of a complex hypersurface $S\subset\CC{n+1}$ if and only if $\omega_0|_{\tilde S}=0$.

Further, we observe that, in the case of a {\em real-analytic} hypersurface $M\subset\CC{n+1}$, 
the associated system \eqref{wzkzl} amounts to an integrable holomorphic 
$n$-distribution in $E$ given by the condition:
\begin{equation}\Label{distrib}
\omega=(\omega_0,\omega_1,...,\omega_n)=0,
\end{equation}
where 
\begin{equation}\Label{omegak}
\omega_k:=d\xi_k-\sum_1^n\Phi_{kl}dz_l,\,\,k=1,...,n
\end{equation} 
(in the sense that the leaves of the foliation $\mathcal F$ determined by \eqref{distrib} are precisely the lifts to $E$ of graphs of solutions for \eqref{wzkzl}).  

In accordance with the latter observation, let us denote the smooth symmetric matrix function \eqref{Phi} (or respectively \eqref{Phi1}) associated with a smooth hypersurface $M$ satisfying Condition E by $\Phi$, and consider the associated complex valued differential $1$-forms $\omega_1,...,\omega_n$, defined by \eqref{omegak}. 
In view of Condition E, the function $\Phi$ and hence all the $1$-forms $\omega_0,\omega_1,...,\omega_n$ extend holomorphically to the pseudoconvex side $\widehat U^+$ of $\widehat M$, defining there a holomorphic $n$-distribution $D$.  Alternatively, $D$ is spanned by the $n$ holomorphic vector fields 
\begin{equation}\Label{Lj}
L_j:=\frac{\partial}{\partial z_j}+\xi_j\frac{\partial}{\partial w}+\sum_1^n \Phi_{sj}\frac{\partial}{\partial \xi_s}.
\end{equation}

We then have

\begin{proposition}\Label{integrable}
The distribution $D$ in $\widehat U^+$ is integrable.
\end{proposition}
\begin{proof}
Integrability of the distribution $D$ amounts to the conditions
\begin{equation}\Label{integrability}
L_j\Phi_{kl}-L_k\Phi_{jl}=0,\quad j,k,l=1,...,n, 
\end{equation}
where $L_j$ are as in \eqref{Lj}.
(In terms of the system \eqref{wzkzl}, conditions \eqref{integrability} mean simply the symmetry of third order derivatives of $w$ in all indices). In view of the Condition E, the left hand side in \eqref{integrability} extends smoothly to the real hypersurface $\widehat M\subset E$. We claim that the latter extension vanishes on the totally real manifold $M_J$. Indeed, the fact of vanishing of the left hand side in \eqref{integrability} when restricted on $M_J$ is nothing but the symmetry of the third order jet of formal Segre varieties of $M$ in all indices, which proves the claim. Since $M_J$ is totally real of dimension $2n+1$, this implies that the left hand side in \eqref{integrability} vanishes identically in $\widehat U^+\cup \widehat M$, as required.

\end{proof}

\autoref{integrable} implies the existence of an $n$-dimensional holomorphic foliation $\mathcal F$ in $\widehat U^+$ generated by $D$. We specify that by {\em leaves} of the foliation $\mathcal F$ we mean maximal connected components of integral submanifolds of $\mathcal F$.
\begin{definition} In what follows we call $\mathcal F$ 
{\em the Segre type foliation in $\widehat U^+$.}
\end{definition}

\subsection{Changing the complex structure on the pseudoconcave side of $M$}

In this section, we show that the pseudoconcave side $U^-$ of $M$ can be interpreted as the space of leaves for the Segre foliation $\mathcal F$ constructed above, and this endows $U^-$ with a {\em different} (integrable) complex structure, which is smooth up to $M$ and which induces on $M$ a CR-structure coincident with the initially given CR-structure on $M$ induced from $\CC{n+1}$.  
We construct the desired complex structures 
in multiple steps discussed in detail below.

\medskip

{\noindent \bf Step I.} We recall that the distribution $D$ above can be, according to Condition E, smoothly extended to $\widehat U^+\cup\widehat M$ as a function valued in the complex Grassmannian  $\mbox{Gr}(n,E)$. (Note though that this extension is {\em not} everywhere tangent to  $\widehat M$!). Furthermore, we note that $M_J$ is precisely the locus of points in $\widehat M$, for which the value of the extension of $D$ {\em is tangent} (and hence complex tangent) to $\widehat M$. This follows directly from the construction of $M_J,\widehat M$.


\medskip

{\noindent \bf Step II.} 
Our next goal is to show that the space of leaves of the Segre foliation 
$\mathcal F$ is a  smooth manifold in its natural (quotient) topology, which can be furthermore extended to a smooth manifold with boundary $M_J$.

We will make use of the following
\begin{proposition}\Label{extension}
Let $U\subset\RR{m}$ be a neighborhood of the origin and $M\subset U$ a smooth strictly convex hypersurface through the origin and, furtheremore, one has $T_0M=\{x_m=0\}$ and the 
second fundamental form of $M$ at $0$ equals to 
\begin{equation}\Label{ksum}
dx_1^2+...+dx_{k}^2
\end{equation} for some $1\leq k <m$. Let $U^+$ be the convex side of $M$, $U^-$ the concave side of $M$, and $D$ a smooth $k$-dimensional integrable distribution in $U^+$. Assume that 

\smallskip 

(i) $D$ extends to $M$ smoothly (as a function valued in $\mbox{Gr}(k,\RR{m})$);

\smallskip

(ii) The $k$-plane $D_0$ at $0$  is spanned by $\frac{\partial}{\partial x_1},...,\frac{\partial}{\partial x_k}$ (in particular, $D_0\subset T_0 M$) and, moreover, the $\frac{\partial}{\partial x_{k+1}},...,\frac{\partial}{\partial x_{m}}$ components of 
some collection of vector fields spanning $D$ have zero linear parts 
at the origin. 

\smallskip

\noindent Then, after possibly changing the neighborhood $U$, the distribution $D$ extends to a smooth integrable distribution in $U$. Furthermore, if $D$ is given by
smooth up to $M$,  (pointwise) linearly independent and commuting vector fields $X_1,..,X_k$  in $U^+$, then these vector fields can be extended smoothly to $U$ in such a way that the extensions still commute. The foliation $\mathcal F$ in $U^+$ generated by $D$ extends 
therefore to a smooth $k$-dimensional foliation $\mathcal F'$ in $U$, in the sense that 
each leaf of $\mathcal F$ is an open subset of a unique leaf of $\mathcal F'$ and each intersection of a leaf of $\mathcal F'$ with $U^+$ is connected.
\end{proposition}
\begin{proof}
We prove the proposition by induction. 

For $k=1$, we choose a vector field $X$ generating $\mathcal F$. We split $x=(x_1,\tilde x)$. In view of (ii), we may assume 
\begin{equation}\Label{conv}
X=p\frac{\partial}{\partial x_1}+q\frac{\partial}{\partial \tilde x},\quad p(0)=1,\,\,q(0)=0,\,\,\frac{\partial q}{\partial x_1}=0.
\end{equation}
Consider a smooth extension  $X'$ of $X$ to $U$ (possible by (i)). Then \eqref{conv} immediately implies that the orbit of $X'$ at $0$ has the form 
$$x_1=t,\quad \tilde x=O(t^3), \quad t\in(-\epsilon,\epsilon).$$ 
Now if, for example, $m=2$ (so that $\tilde x=x_2$), we consider the defining equation $x_2=\psi(x_1)$ of $M$ as well as the definining equation $x_2=\phi_0(x_1)$ of the orbit. Then \eqref{ksum} implies that $(\psi-\phi)''>0$ on $(-\epsilon,\epsilon)$ for small elough $\epsilon$,  that is why the  part of the orbit lying inside $U^+=\{x_2>\psi\}$ is the set of negative values of a convex function, i.e. the intersection of the orbit with $U^+$ is {\em connected}. By continuity, after shrinking possibly $U$, the same argument applies for all points nearby $0$, and this proves the proposition for $k=1$ and $m=2$. The case $k=1$ and $m>2$ can be reduced to the previous one by considering the intersection with a $2$-dimensional surface containing the orbit and the $x_m$ coordinate axis. We leave the details to the reader.

We now proceed with the induction step. 
We choose $k$ linearly independent smooth vector fields $X_1,...,X_{k}$ spanning the distribution $D$ in $U^+$ and extending smoothly to $M$. Following a proof of the Frobenius theorem (e.g. \cite{book}), it is not difficult to show that we can choose $X_1,...,X_k$ to be furthermore {\em commuting}. Indeed, we have:
$$X_i=\sum_{j=1}^m \alpha_{ij}\frac{\partial}{\partial x_j}$$
for smooth up to $M$ functions $\alpha_{ij}$ in $U^+$
($(x_1,...,x_m)$ are the coordinates in $\RR{m}$). Since $X_i$ are linearly independent, we can assume without loss of generality that the matrix 
$$\left(\alpha_{ij}\right)_{i,j=1}^k$$
is invertible in $U^+\cup M$, and consider the inverse smooth matrix $(\beta_{ij})$. Then it is straightforward to check that the vector fields 
$$\sum_{j=1}^k\beta_{ij}X_j,\quad 1\leq i\leq k,$$ obviously spanning the same distribution in $U^+$ and smooth up to $M$, in fact in addition {\em commute}. 

The latter allows us to consider the {\em integrable} distribution generated by the commuting vector fields $X_1,...,X_{k-1}$. Applying then the induction assumption, we obtain smooth {\em commuting} extensions $X_1',..,X_{k-1}'$ of $X_1,..,X_{k-1}$ and hence a $(k-1)$-dimensional integrable distribution  in $U$. The foliation $\mathcal X$ given by this distribution has the property that its  leaves can have only connected intersections with $U^+$.     After that, let us perform (after possibly shrinking $U$) a smooth in $U$ diffeomorphism, 
tangent to the identity at $0$ and
transforming the vector fields $X_1',...,X_{k-1}'$ to $\frac{\partial}{\partial x_1},...,\frac{\partial}{\partial x_{k-1}}$ respectively (the latter is possible since the vector fields commute). Now all the leaves of $\mathcal X$ become parallel to the $(x_1,...,x_{k-1})$-plane. We keep the same notation for $U,U^\pm$ in the new coordinates. 
Then the vector field $X_k$ (defined so far in the closure of $U^+$) 
 have the form
\begin{equation}\Label{YY}
X_k=a_1(x_k,...,x_m)\frac{\partial}{\partial x_1}+\cdots+a_m(x_k,...,x_m)\frac{\partial}{\partial x_m}
\end{equation} 
 where
we use the commutativity $[X_j,X_k]=0,\,j=1,...,k-1$). 

Consider the orthogonal projection $\Omega$ of the closure of $U^+$   onto the $(x_k,...,x_m)$-plane. Then the leaves of $\mathcal X$ intersecting the closure of $U^+$ project onto a single point in $\Omega$.  As follows from \eqref{YY}, the vector field $X_k$ is {\em constant} on each leaf of $\mathcal X$, which allows to extend $X_k$ constantly along each leaf intersecting the closure of $U^+$.  In view of the above, this gives a smooth function on $\Omega$, which we first extend smoothly to a neighborhood of the origin in the $(x_k,...,x_m)$-plane,  and then again constantly along each leaf of $\mathcal X$.
Since the intersection of each leaf of $\mathcal X$ with the closure of $U^+$ is connected, the extension obtained is well defined. 
We thus are able to extend $X_k$ smoothly to a full neighborhood of the origin still being constant on each leaf of $\mathcal X$.
 In view of the latter property, the extended vector field $X_k'$ also
 satisfies $[X_j',X_k']=0,\,1\leq j\leq k-1$.

In summary, we obtain an integrable distribution in $U$ spanned by $X_1',...,X_k'$. For the respective foliation $\mathcal F'$, each leaf of $\mathcal F$ is clearly contained in that of $\mathcal F'$. It remains to show that intersections of leaves of $\mathcal F'$ with $U^+$ are connected. This however can be seen from (ii) by an argument identical to the one in the $1$-dimensional case. 
The proof is complete.
\end{proof}

We now apply \autoref{extension} to the situation of the
strictly pseudoconvex hypersurface $\widehat M\subset E$, its neighborhood $\widehat U$ and the distribution $D$ (considered as a real distribution).  For that, we perform a biholomorphic (in fact polynomial) coordinate change mapping the origin and the tangent plane $\im w=0$ onto themselves, and removing holomorphic quadratic terms from the formal Taylor expansion of $M$ in the origin
such that property \eqref{ksum} holds for $\widehat M$.    
In fact, in such coordinates $M$ becomes approximated by a quadric to order $2$ at the origin: 
\begin{equation}\Label{approx}
\im w=Q(z,\bar z)+O(3),
\end{equation}
where $Q$ is a positive definite Hermitian form (we assume it to be simply sum of squares) and $O(3)$ stand for terms of degree $3$ and higher. In case $M$ is already this quadric itself, $D$ is spanned by the (real parts of) the vector fields  
$$L_j:=\frac{\partial}{\partial z_j}+\xi_j\frac{\partial}{\partial w},$$
so that (ii) is satisfied. However, terms of order $3$ and higher in \eqref{approx} do not effect (ii), and  
this finally shows that $\widehat M$ satisfies all the conditions of \autoref{extension}.

We end up with a smooth extension of the distribution $D$ to a full neighborhood of the origin in $E$ in such a way that it is still integrable and defines a foliation $\mathcal F'$ extending $\mathcal F$, in the sense that each leaf of $\mathcal F$ is an open subset of some leaf of $\mathcal F'$ and, furthermore, leaves of $\mathcal F'$ have only connected intersections with $\widehat U^+$ (so that each leaf of $\mathcal F$ is contained in exactly one leaf of $\mathcal F'$).

\bigskip

{\noindent \bf Step III.} 
Based of the outcome of Step II, we are finally able 
to endow the space of leaves of $\mathcal F$
with the structure of a smooth $(2n+2)$-manifold with boundary $M_J$ in the natural (quotient) topology. 
Indeed, first note that the tangent plane at $0$ to the leaf of $\mathcal F'$ through $0$ is $w=0,\,\xi_j=0$ (as follows from the definition and the initial normalization of $M$). Hence, the $(2n+2)$-plane 
\begin{equation}\Label{z0}
\{z=0\}
\end{equation}  
has the property that all the leaves of (the ambient foliation) $\mathcal F'$ intersect it transversally at single points 
(after possibly shrinking $\widehat U$). 
Thus the space of leaves of $\mathcal F'$ can be identified with a domain in \eqref{z0} (viewed as $\RR{2n+2}$). Accordingly, the space of leaves of $\mathcal F$ is an open connected subset $G$ of the latter domain (since it is given by the condition of having nonempty intersection with the domain $\widehat U^+$). 
This gives a structure of a smooth manifold on the space of leaves of $\mathcal F$ (in its natural quotient topology)! 

It remains to show that the space of leaves of  $\mathcal F$ has, furthermore, a structure of a smooth $(2n+2)$-manifold with a boundary. 
Indeed, as follows from the discussion in Steps I and II, any leaf of $\mathcal F'$ intersecting the closure of $\widehat U^+$  either intersects $\widehat M$ transversally and consequently intersects the open part $\widehat U^+$, or  intersects the boundary $\widehat M$ only at a point in $M_J$, or is contained in $\widehat U^+$ (the last possibility in fact does not occur, but we do not need this fact). In this way, {\em $M_J$ can be identified with the subset of leaves of $\mathcal F'$   intersecting the boundary $\widehat M$ but not the open part $\widehat U^+$}. Thus, we are looking for the set of leaves of $\mathcal F'$ which are tangent at some point to $\widehat M$. The latter set is a smooth $(2n+1)$-submanifold in $\RR{2n+2}$ (which is in fact the boundary of the above open set $G$ of leaves of $\mathcal F'$ intersecting $\widehat U^+$). Indeed, perform a local diffeomorphism (near the origin in $E\sim\RR{4n+2}$) with the identity linear part in such a way that in the new local coordinates the leaves become "horizontal" $2n$-planes (that is, they are given by $$x_j=c_j,\,j=2n+1,...,4n+2,$$ where $c_j$ are constant), and the hypersurface $\widehat M$ becomes
$$
	x_{4n+2}=\varphi(x_1, \ldots ,x_{4n+1})
$$ 
for a smooth function $\varphi$. Note that, in particular, $\varphi$ has nondegenerate at the origin Hessian in $x_1,...,x_{2n}$ (as follows, for example, from \eqref{approx}).  Now the condition that a leaf is tangent to $\widehat M$ at some point is:
\begin{equation}\Label{tang}
c_{4n+2}=\varphi(x_1,...,x_{2n},c_{2n+1},...,c_{4n+1}),\quad \varphi_{x_j}(x_1,...,x_{2n},c_{2n+1},...,c_{4n+1})=0,\,\, j=1,..,2n.
\end{equation}       
Solving the last $2n$ equations of \eqref{tang} in $x_1,...,x_{2n}$ by the implicit function theorem and substituting the result into the first equation of \eqref{tang}, we obtain a smooth hypersurface 
$$c_{4n+2}=\psi(c_{2n+1}, \ldots ,c_{4n+1})$$
in $\RR{2n+2}$ (endowed with the coordinates $c_{2n+1},...,c_{4n+1}$) for an appropriate smooth function $\psi$ with $\psi(0)=d\psi(0)=0$, as desired.

\bigskip

We thus have proved the following

\begin{proposition}\Label{manifold}
The leaf space of the foliation $\mathcal F$ is a smooth $(2n+2)$-manifold in its natural (quotient) topology. Furthermore, it can be regarded as a smooth $(2n+2)$-manifold with boundary $M_J$. The (germ at the origin of) the upper half space 
$$\bar H=\{(x_{2n+1},...,x_{4n+2}):\,\,x_{4n+2}\geq 0\}$$ serves as a coordinate chart for it. 
\end{proposition}


\medskip

{\noindent \bf Step IV.}
We denote the leave space from \autoref{manifold} by $\mathcal U$, and the respective manifold with boundary by $\overline{\mathcal U}$. As was mentioned above, $\mathcal U$ represents, in a certain sense, the pseudoconcave side $U^-$ endowed however with a {\em different} (still integrable) complex structure. 

We 
next note that the leave space $\mathcal U$ is, on the other hand, the quotient topological space $\widehat U^+/\mathcal F$, and this means (if read together with \autoref{manifold}) that $\mathcal U$ has also a structure of a {\em complex} manifold with boundary. We emphasize at this point that, exclusively for the purpose of obtaining the right complex structure on $\widehat U^+/\mathcal F$,

\smallskip

 {\em we shall change the complex structure on 
 $\mathcal U$  to its conjugate structure.} 

\smallskip 
 This is related to the antiholomorphic dependence of Segre varieties on their parameters, for a real-analytic hypersurface. Thus we obtain a  smooth integrable complex structure on $\mathcal U$, which extends further to a smooth CR-structure on $\partial\mathcal U$. The boundary $\partial\mathcal U$ is naturally diffeomorphic to $M_J$ and hence to $M$. We will show later that the two CR-structures on $M$ (the one coming from the quotient space and the one induced from the embedding into $\CC{n+1}$) in fact agree with each other together with all higher order derivatives.

\medskip

{\noindent \bf Step V.} 
We now have to take into consideration the ``one sided Segre varieties of $M$'', that is, images of leaves of the foliation $\mathcal F$ under the projection map $\pi$, as in \eqref{PI}. This gives us the family
$$\mathcal S^+:=\Bigl\{\pi(T)\Bigr\}_{T\in \mathcal F}.$$
Note that, as the extension construction in Step II above shows, all the leaves in $\mathcal F$ 
(after
possibly 
shrinking
the basic neighborhood $\widehat U$) are open subsets of graphs of smooth functions of the kind $w=w(z),\,\xi=\xi(z)$ with
 connected intersections with $\widehat U^+$. In this way, all elements of $\mathcal S^+$ are $n$-dimensional complex submanifolds in $U^+$, and we conclude that 
{\em $\mathcal S^+$ is an $(n+1)$-dimensional anti-holomorphic family 
of 
pairwise transverse 
complex $n$-dimensional submanifolds in $U^+$ of the form $w=w(z)$}. 
(Transversality here means that no two elements of $\mathcal S^+$ are tangent at a point $p\in U^+$). 
The anti-holomorphic parametrization of $\mathcal S^+$ here comes from the integrable complex structure on the space of leaves $\mathcal U$.  

We note that $\mathcal U$ itself is endowed with a natural $(n+1)$-dimensional anti-holomorphic family of $n$-dimensional complex submanifolds as follows. We fix a point $p\in U^+$ and consider the set $S_p$ of all the manifolds from $\mathcal S^+$ passing through $p$ as a subset of $\mathcal U$. Following the geometric interpretation in the real-analytic case, we call $S_p$ {\em the Segre variety of $p$}. The structure of $S_p$ becomes particularly clear when considering the foliation $\mathcal F$: then the set of all elements of $\mathcal S^+$ passing through $p$ lifts to the set of all fibers in $\mathcal F$ intersecting the fiber $\pi^{-1}(p)$ of the bundle $E$. In this way, we easily see, from the construction of the manifold $\mathcal U$,  that each $S_p$ can be identified via the $1$-jet map
with the fiber $\pi^{-1}(p)$ and thus is a complex $n$-dimensional submanifold in $\mathcal U$ (with respect to the above described complex structure on $\mathcal U$), as required. We denote the resulting family of submanifolds in $\mathcal U$ by $\mathcal S^-$ (it becomes an anti-holomorphic family parameterized by $U^+$). 

For completeness of the picture, we also call, for each $p\in\mathcal U$, the respective leaf $T\in \mathcal S^+$ its {\em Segre variety} and denote the latter one by $S_p$. We then obtain the following familiar symmetry property:
$$p\in S_q \Leftrightarrow q\in S_p, \quad p\in U^+,\,\,q\in \mathcal U.$$ 

\medskip

{\noindent \bf Step VI.} We recall that the boundary manifold $\partial\mathcal U$ is naturally diffeomorphic to the initial CR-manfold $M$, as follows from the construction of $\overline{\mathcal U}$.  We further extend (locally near $0$) the latter diffeomorphism smoothly to a diffeomorphism between the manidold with boundary $\overline{\mathcal U}$ and the pseudoconcave side $U^-$ of $M$ (which is possible since both are 
manifolds
of equal dimension
 with boundary). 
 We end up with a smooth manifold $U$ decomposed as a union of two manifolds with boundary:
$$U=(U^-\cup M)\cup (U^+\cup M),$$
where both $U^-$ and $U^+$ are endowed with their individual complex structures (for $U^-$ this is the integrable structure induced from $\mathcal U$ and for $U^+$ this is the standard complex structure induced from $\CC{n+1}$). Moreover, both structures admit a smooth extension to the boundary and  induce boundary CR-structures on $M$. Our goal is now to show that the two structures (considered for the moment as $(2n+2)\times (2n+2)$ matrices) agree on $M$ (together with all derivatives). In particular, they define a smooth structure in a full neighborhood of the origin, and the two induced CR-structures on $M$ coincide.   

For doing so, let us fix $p\in M$ and  the respective point $\tilde p\in\partial\mathcal U$. We observe the following: 
all 
data required for computing the boundary value at  $\tilde p$ of the complex structure on $\mathcal U$ comes from the $2$-jet of $M$ at $p$ (as follows from our construction). Similarly, for computing the $k$-jet at $p$ of the limit 
of the complex 
structure we just need to know the $(k+2)$-jet of $M$ at $p$. Since $M$ can be approximated to any order by a real-analytic hypersurface, we conclude that it suffies to show that for a {\em real-analytic} hypersurface $M$ the two above structures coincidence and define a real-analytic (in particular smooth) structure in a neighborhood of $p$. 

If now $M$ is real-analytic, then, as a well know fact (e.g. \cite{ber}), after choosing an appropriate neighborhood $U$ of $p$, we have the property that a Segre variety of a point $q\in U$ intersects the pseudoconvex side $U^+$ iff $q$ lies in the pseudoconcave side $U^-$ of $M$. Thus, if $M=\{\rho(Z,\bar Z)=0\}$ near $p$, then the above manifold $\mathcal U$ consists of Segre varieties 
$\{\rho(Z,\bar q)=0\}$ 
with $q\in U^-$, and thus  can be identified with $U^-$ 
with the standard complex structure on it, 
while its boundary  can be identified  with $M=\{\rho(q,\bar q)=0\}$ with the standard CR-structure on it. 
This immediately yields the desired property.

\medskip

\subsection{End of proof of the main result}

In this section, we complete the proof of the main result. 

\smallskip

Recall that, as an outcome of Step VI, we obtain a smooth manifold $U$ endowed with an {\em integrable} smooth complex structure $J$ (the integrability follows from that on both $U^-$ and $U^+$ and thus, by continuity, at points in $M$ as well). By the Newlander-Nirenberg theorem \cite{nirenberg}, there is a smooth diffeomorphism $\chi$ of $(U,J)$ (preserving the origin), mapping $U$ onto a neighborhood of the origin
and transforming the above complex structure $J$ on $U$ into the standard complex structure $J_{st}$ in $\CC{n+1}$ (that is, $\chi$ is a $(J,J_{st})$-biholomorphism). The resulting  smooth strictly pseudoconvex image of $M$ we still denote by $M$, and the pseudoconvex and the pseudoconcave sides of it respectively we still denote by $U^\pm$.

We shall consider now   the above families $\mathcal S^\pm$ on $U^\pm$ respectively, after applying the diffeomorphism $\chi$. 
We recall that elements of both families $\mathcal S^\pm$ are $J$-invariant and have the transversality property, hence they become 
{\em 
families of holomorphic curves} on $U^\pm$ respectively. This allows us to consider, in the same fashion as in Section 2, their lifting to the space of $1$-jets $J^{1,n}$, and this results in two foliations  defined in some  domains 
$\widehat U^\pm\subset E$ with $\pi(\widehat U^\pm)=U^\pm$ respectively  (here $E$ is again the affine subset of the bundle of $1$-jets of complex hypersurfaces). In particular, we may consider the respective holomorphic direction fields defined in the same domains. As follows from the above, these two directions fields extend smoothly to $M_J$. 

Importantly, we can {\em not} conclude at this step that the domain $\widehat U^-$ has the form $U^-\times\Omega$, where $\Omega$ is an open neighborhood of the origin in $\CC{n}$. This is because $1$-jets of elements of $\mathcal S^-$ passing through a point $p^-\in U^-$ do not cover a full neighborhood of the origin of a uniform size. To see the latter, we note that the 
"pencil" of Segre varieties passing through $p^-\in U^-$ is the union of 
"pencils"
of Segre varieties from $\mathcal S^+$ at points belonging to the Segre variety $S_{p^-}\in\mathcal S^+$. Unlike the situation in the real-analytic case, $S_{p-}$ is defined so far {\em only} as a variety in $U^+$, 
which is not a full neighborhood of $0$. 
(In the real-analytic case, such "one-sided" Segre varieties extend analytically across $M$ and become varieties in a uniform neighborhood of $0$). 
That is why possible jets of Segre varieties through $p^-$ in our construction form a "smaller" set in the space of $1$-jets (compared to the real-analytic case) and thus do not give a uniform neighborhood of the origin.
 (However, the domain $\widehat U^+$ does have the desired form $U^+\times\Omega$ for an open neighborhood $\Omega$ of the origin). 

To overcome the latter difficulty, we use Webster's considerations from \cite{websterrefl} to conclude that, after possibly 
shrinking the neighborhood $U$:

\medskip

\noindent (i) There exists a wedge $W^+\subset\widehat U^+$ with the totally real edge $M_J$;

\smallskip

\noindent (ii) The {\em reflection map} $\tau$ defined as
\begin{equation}\Label{tau}
\tau(z, T_z S_\zeta)=(\zeta,T_\zeta S_z), \quad z\in\widehat U^\pm,\,\, \zeta\in S_z
\end{equation} 
is anti-holomorphic, well defined in $\widehat U^\pm$, and smooth up to $M_J$;

\smallskip

\noindent (iii) The image $\tau(W^+)$ contains a wedge $W^-\subset\widehat U^-$, and $\tau$ satisfies
$$\tau\circ\tau=\mbox{Id}$$
(thus $\tau$ is an anti-holomorphic involution);

\smallskip  

\smallskip

\noindent (iv) $\tau|_{M_J}=\mbox{Id}$.

\medskip

To prove (i)-(iii) we note that, even though \cite{websterrefl} deals with the real-analytic case, the proof of (i)-(iii) in \cite{websterrefl} is based solely on the approximation of $M$ by a quadric to order $3$ and the implicit function theorem subsequently, that is why it can be essentially word-by-word repeated in our situation. We leave the details here to the reader. Statement (iv) follows directly from the construction in Section 4.2 (alternatively, in can be viewed from (i)-(iii) read together).

We now use the edge-of-the-wedge theorem (e.g. \cite{rosay}) to conclude that $\tau$ extends anti-holomorphically to a full neighborhood of the origin in $E$ (still being an involution, by uniqueness). For simplicity, we denote the latter neighborhood by $\widehat U$. 

Let us consider finally (locally near the origin) the fixed point set $M'\subset U$ of $\tau$. In view of the fact that $\tau$ is an anti-holomorphic involution, $M'$ is a {\em real-analytic} totally real submanifold in $\widehat U$ of dimension $2n+1$. In view of (iv), we conclude that $M'=M_J$, i.e. both $M_J$ and $M$ are {real-analytic}. The proof
of 
\autoref{main}
is complete.

\qed

\end{document}